\documentclass[USenglish]{article}	

\usepackage[utf8]{inputenc}				
\usepackage[big,online]{dgruyter}	
\usepackage{lmodern} 
\usepackage{microtype}
\usepackage[numbers,square,sort&compress]{natbib}

\theoremstyle{dgthm}
\newtheorem{theorem}{Theorem}

\newtheorem{lemma}{Lemma}

\theoremstyle{dgdef}
\newtheorem{definition}{Definition}

\newtheorem*{thm}{Theorem A } 
\newtheorem*{theo}{Theorem B}

\begin{document}

	\articletype{Research Article}
	\received{Month	DD, YYYY}
	\revised{Month	DD, YYYY}
  \accepted{Month	DD, YYYY}
  \journalname{De~Gruyter~Journal}
  \journalyear{YYYY}
  \journalvolume{XX}
  \journalissue{X}
  \startpage{1}
  \aop
  \DOI{10.1515/sample-YYYY-XXXX}

\title{The properties of general Fourier partial sums of functions $f \in C_L$}
\runningtitle{The properties of general Fourier partial sums}

\author*[1]{Giorgi Tutberidze}
\author[2]{Vakhtang Tsagareishvili}
\author[2]{Giorgi Cagareishvili} 
\runningauthor{F.~Author et al.}
\affil[1]{\protect\raggedright 
The University of Georgia, Institute of Mathematics, 77a Merab Kostava St, Tbilisi 0128, Georgia and Ivane Javakhishvili Tbilisi State University, Faculty of Exact and Natural Sciences, Chavchavadze str. 1, Tbilisi 0128, Georgia, e-mail: g.tutberidze@ug.edu.ge, g.tutberidze@tsu.ge}
\affil[2]{\protect\raggedright 
Ivane Javakhishvili Tbilisi State University, Department of Mathematics, Faculty of Exact and Natural Sciences, Chavchavadze str. 1, Tbilisi 0128, Georgia, e-mail:cagare@ymail.com, giorgicagareishvili7@gmail.com}

	
\abstract{In this paper, we investigate the convergence properties of Fourier partial sums associated with general orthonormal systems, focusing on functions that belong to specific differentiable function classes. While classical Fourier analysis has extensively studied trigonometric systems, our approach considers a broader class of orthonormal systems, including those adapted to weighted function spaces or arising from orthogonal polynomials. The primary objective is to analyze how the smoothness and differentiability of the function \( f \) affect the rate and nature of convergence of its Fourier partial sums. We derive estimates for the approximation error in various norms and establish sufficient conditions under which uniform or pointwise convergence occurs. In particular, we highlight how differentiability constraints on \( f \) can lead to sharper convergence results than those available for general \( L^2 \)-functions. Furthermore, we explore the impact of specific system properties, such as localization and boundedness, on the summation behavior. Several illustrative examples are provided, demonstrating the theoretical findings for commonly used orthonormal systems. Our results contribute to the deeper understanding of spectral approximations and have potential applications in numerical analysis, signal processing, and the theory of function spaces.}


\maketitle

\section{Introduction}\label{sec1}

\qquad To maintain the logical flow of our discussion and ensure the clarity of our main proofs, we have compiled all relevant notations, definitions, and preliminary concepts in Section~2.

S.~Banach famously demonstrated that good differentiability properties of a function do not necessarily imply the almost everywhere boundedness of its Fourier partial sums with respect to general orthonormal systems (ONS). In this paper, we investigate the behavior of Fourier partial sums with respect to general ONSs when the function \( f \) belongs to certain differentiable function classes. We also show that our main result is, in a specific sense, sharp---see Theorem~\ref{t3}.

Regarding recent applications developed in collaboration with the present authors, we refer the reader to the Ph.D. theses~\cite{harpal} and~\cite{tutberidze}, as well as the references therein. These works explore significant engineering and industrial applications, including topics such as Structural Health Monitoring, Artificial Intelligence, Neural Networks, Signal Processing, Operational Modal Analysis (OMA), and Damage Detection in Bridges, with particular attention given to the remarkable H\aa{}logaland Bridge in Narvik.

An interesting aspect of our investigation is the distinction, noted in the results of Menchov and Banach, between the convergence behavior of general orthonormal series and that of general Fourier series for functions in certain differentiability classes. In the case of general orthonormal series, the convergence primarily depends on the coefficients. However, for Fourier series with respect to a general ONS, the mere fact that \( f \neq 0 \) belongs to a differentiable class does not ensure convergence. Consequently, for the Fourier series of functions in a given differentiability class (e.g., \( \text{Lip}1 \)) to be convergent or for their partial sums to remain bounded, one must impose additional conditions on the elements \( \varphi_n \) of the ONS \( (\varphi_n) \).

A further motivation for this study is to utilize our main results to derive new findings concerning the convergence and divergence of general Fourier series. Related results can be found in~\cite{Edvards, GogoladzeTsagareishvili, GogoladzeTsagareishvili_5, GogoladzeTsagareishvili_6, KolmogorovFomin, PTT, Olevskii, Rademacher, Tandori, Cagareishvili, TsaTutCa, tsatut3, tut1}. See also the book~\cite{PTWbook} and the monographs~\cite{Alexits, KaczmarzSteinhaus, KashinSaakyan}.

The main results---Theorems~\ref{t2} and~\ref{t3}---are stated and proven in Section~3.

\section{Auxiliary definitions and results} \label{sec2}
By $Lip1$ we denote the class of functions $f$ from  $C(0,1)$, for which
$$\|f(x)-f(y)\|=O(h), \text{ when } \max |x-y| \leq h.$$
Let $C_L$ be the class of functions $f$ if 
$\frac{d}{dx}f \in Lip1.$

Suppose that $f\in L_2$ be an arbitrary function and $(\varphi_n)$  be ONS on $[0,1]$. Then the numbers
$$
C_n(f)=\int_0^1 f(x)  \varphi_n (x)\,dx, \quad n=1,2,\dots\,,
$$
are the Fourier coefficients of function $f$ with respect to the system $(\varphi_n)$ and
\begin{equation}\label{eq2}
	\sum_{n=1}^\infty   C_n (f)\varphi_n (x)
\end{equation}
is the Fourier series of this function $f$.

The general partial sum $S_n(x,f)$ of the series \eqref{eq2} is defined as follows
\begin{eqnarray}\label{eq2.2}
	S_n (x,f)=\sum_{k=1}^n  C_k (f)\varphi_k (x) .
\end{eqnarray}
Next, we can define $M_n(x)$ as follows $\left(x \in [0,1]\right)$ 

\begin{equation}\label{eq3}
	M_n(x)=\frac{1}{n} \sum_{i=1}^{n-1}\left|\int_{0}^{\frac{i}{n}}Q_n(u,x)du\right|,
\end{equation}
where
\begin{equation}\label{eq4}
	Q_n (u,x)=\sum_{k=1}^n  g_k (u)\varphi_k (x) 
\end{equation}
and
\begin{eqnarray}\label{eq1.2}
	g_k(u)=\int_{0}^{u} \varphi_k(t)dt. 
\end{eqnarray}

\begin{lemma} (see \cite{Cagareishvili})\label{l1}
	Let $(\varphi_n)$ be an arbitrary ONS on $[0,1]$. Then
	$$
	\frac{1}{n^2}  \sum_{k=1}^n  \varphi_k^2 (x)=O(1)n^{-\frac{1}{2}}  \;\;\text{a.e. on $[0,1]$}.
	$$
\end{lemma}

\textbf{Notation.} Let $G=[0,1] \setminus F,$ where 
\begin{eqnarray*}
	F = \left \{x\in[0,1] \left| \lim\limits_{n\to \infty}\frac{1}{n^2} \sum_{k=1}^{n} \varphi_k^2 (x) = \infty\right. \right \}.
\end{eqnarray*}
It is easy to show that
$$\left|F\right|=0 \text{  and  } \left|G\right|=1.$$

\begin{lemma}(see \cite{Cagareishvili})\label{l2}
	Suppose that
	\begin{equation}\label{eq5}
		B_n (u,x)=\sum_{k=1}^n  \varphi_k (u) \varphi_k (x).
	\end{equation}
	If for some $x\in [0,1]$
	\begin{equation}\label{eq6}
		\limsup_{n\to \infty} \bigg\vert \int_0^1  B_n (u,x)\,du\bigg\vert =+\infty,
	\end{equation}
	then for the function $q(u)=1$ $(u\in [0,1])$
	\begin{eqnarray*}
		\limsup_{n\to \infty} \vert S_n (x,q)\vert =+\infty ,
	\end{eqnarray*}
	where
	\begin{eqnarray*}
		S_n(x,q)=\sum_{k=1}^{n} C_k(q)\varphi_{k}(x) \text{ and } C_k(q)=\int_{0}^{1}q(u)\varphi_{k}(u)du = \int_{0}^{1} \varphi_{n}(u)du.
	\end{eqnarray*}

\end{lemma}
\begin{proof}
	Indeed, as
	$$C_n(q) = \int_{0}^{1}\varphi_{n}(u)du $$
	we get
	\begin{eqnarray*}
		\int_{0}^{1} B_n(u,x)du = \sum_{k=1}^{n}\int_{0}^{1}\varphi_{k}(u)du \varphi_{k}(x)=\sum_{k=1}^{n} C_k(q)  \varphi_{k}(x) = S_n(x,q).
	\end{eqnarray*}
	From the last equality and \eqref{eq6} we conclude that Lemma \ref{l2} is valid.
	
\end{proof}

\begin{lemma} (see \cite{Cagareishvili}) \label{l3}
	For any $i$ $(i=1,\dots,n)$ and $x\in [0,1]$ 
	$$
	\int_{\frac{i-1}{n}}^{\frac{i}{n}} \vert Q_n(u,x)\vert \,du\leq \frac{1}{n} \bigg(\sum_{k=1}^n \varphi_k^2(x)\bigg)^{\frac{1}{2}}
	$$
	holds (see \eqref{eq4}).
\end{lemma}
\begin{proof}
	By using \eqref{eq1.2} and Bessel inequality,
	\begin{eqnarray}
		\label{eq2.7}
		\sum_{k=1}^{\infty} g_k^2(u) = 	\sum_{k=1}^{\infty} \left(\int_{0}^{u} \varphi_k(t)dt\right)^2 \leq \int_{0}^u dt \leq 1	\end{eqnarray}
	where $u \in [0,1].$ Using \eqref{eq2.7}, the Cauchy and H\"older inequalities, we get $\left(i=1, 2, \dots , n\right)$ 
	\begin{eqnarray*}
		\int_{\frac{i-1}{n}}^{\frac{i}{n}} \vert Q_n(u,x)\vert \,du &\leq& \frac{1}{\sqrt{n}} \left(	\int_{\frac{i-1}{n}}^{\frac{i}{n}} Q_n^2(u,x) \,du\right)^{\frac{1}{2}} = \frac{1}{\sqrt{n}} \left(	\int_{\frac{i-1}{n}}^{\frac{i}{n}} \left(\sum_{k=1}^{n}g_k(u)\varphi_{k}(x)\right)^2 \,du\right)^{\frac{1}{2}} \\
		&\leq&	\frac{1}{\sqrt{n}} \left(	\int_{\frac{i-1}{n}}^{\frac{i}{n}} \sum\limits_{k=1}^{n}g_k^2(u)du \sum_{k=1}^n\varphi_{k}^2(x) \right)^{\frac{1}{2}} \leq \frac{1}{n}\left(\sum_{k=1}^n\varphi_{k}^2(x)\right)^{\frac{1}{2}}.
	\end{eqnarray*}
	The Lemma \ref{l3} is proved.
	
\end{proof}

\begin{lemma}(see \cite{Cagareishvili}) \label{l4}
	Let $(\varphi_n )$ be an ONS on $[0,1]$ and $f\in C_L$, then
	\begin{equation}\label{eq8}
		S_n (x,f)=f(1) \int_0^1  B_n (u,x)\,du-\int_0^1  f' (x) Q_n (u,x)\,du  .
	\end{equation}
\end{lemma}
\begin{proof}
	Integrating by parts, we obtain
	$$
	C_n (f)=\int_0^1  f(u) \varphi_k (u)\,du =f(1) \int_0^1  \varphi_k (u)\,du -\int_0^1  f' (u) g_n (u)\,du .
	$$
	Therefore,
	\begin{align}
		S_n (x,f) & =\sum_{k=1}^n  C_k (f)\varphi_k (x) \nonumber \\
		& =f(1) \int_0^1 \sum_{k=1}^n  \varphi_k (u) \varphi_k (x)\,du -\int_0^1  f'(u)\sum_{k=1}^n  g_k (u)\varphi_k (x)   \,du \nonumber \\
		& =f(1) \int_0^1  B_n (u,x)\,du-\int_0^1  f' (u) Q_n (u,x)\,du   . \label{eq9}
	\end{align}
	From \eqref{eq9} we derive \eqref{eq8}.
	
	The Lemma \ref{l4} is proved.
	
\end{proof}

\begin{definition}\label{d2}
	Let $E(\varphi )$ be a set of any functions $f$ with
	\begin{equation}\label{eq10}
		\vert S_n (x,f)\vert =O(1)
	\end{equation}
	at the point $x\in [0,1]$.
	\\
	
\end{definition}

\begin{lemma} \label{l5}
	Suppose that $q\in E(\varphi),$ $x\in G$ $(q(u)=1, \ u\in [0,1])$ and
	\begin{eqnarray*}
		Q_n(u,x) = \sum_{k=1}^{n} g_k(u) \varphi_k(x).
	\end{eqnarray*}
	If for $x \in [0,1]$
	\begin{eqnarray}
		\label{eq2.10}
		\limsup_{n \to \infty} \left|\int_{0}^{1} Q_n(u,x)du\right| = + \infty,
	\end{eqnarray}
	then for the function $p(u)=u$ $(u \in[0,1])$
	\begin{eqnarray*}
		\limsup_{n \to \infty} \left|S_n(x,p)\right| = + \infty.
	\end{eqnarray*}
\end{lemma}

\begin{proof}
	In equation \eqref{eq9} we suppose that $f = p = u$ $(u \in  [0, 1]),$ we receive
	\begin{eqnarray}
		S_n (x,p) = \int_0^1  B_n (u,x)\,du-\int_0^1 Q_n (u,x)\,du   . \label{eq2.11}
	\end{eqnarray}
	According to the condition of this lemma (see Lemma\ref{l2})
	\begin{eqnarray*}
		\left|\int_0^1  B_n (u,x)\,du\right| = O(1).
	\end{eqnarray*}
	Consequently, using \eqref{eq2.10} and \eqref{eq2.11} we get
	\begin{eqnarray*}
		\limsup_{n \to \infty} \left|S_n(x,p)\right| = + \infty.
	\end{eqnarray*}
	Lemma \ref{l5} is proved.	
	
\end{proof}

\begin{thm}[S. Banach \cite{Banach}] \label{tA}
	For any $f\in L_2$ $(f\not\sim 0)$ there exists an ONS $(\varphi_n )$ such that
	\begin{eqnarray*}
		\limsup_{n\to \infty} \vert   S_n (x,f)\vert =+\infty \ \ \ \text{a.e. on } [0,1].
	\end{eqnarray*}	
	
\end{thm}

\begin{theo}[see \cite{2}] \label{tB}
	If   $f, F\in L_2$, then
	\begin{align}
		\int_0^1  f'(u)F(x)\,dx & =n\sum_{i=1}^{n-1}  \int_{\frac{i-1}{n}}^{\frac{i}{n}} \left(f(x)-f\left(x+\frac{1}{n} \right)\right)\,dx \int_0^{\frac{i}{n}}  F(x)\,dx \nonumber \\
		& \qquad  +   n\sum_{i=1}^{n-1}  \int_{\frac{i-1}{n}}^{\frac{i}{n}} \int_{\frac{i-1}{n}}^{\frac{i}{n}} \left(f(x)-f(u)\right)\,duF(x)\,dx \nonumber \\
		& \qquad  +n\int_{1-\frac{1}{n}}^1 f(x)\,dx\int_0^1  F(x)\,dx.\label{eq11}
	\end{align}
\end{theo}

\section{The main problems} \label{sec3}

From Theorem A it follows that even for function $g(x)=1$ $(x\in [0,1])$ there exists an ONS $(\varphi_n)$ such that
$$
\limsup_{n\to \infty} \vert S_n (x,g)\vert =+\infty
$$
a.e. on $[0,1]$.

\begin{theorem} \label{t2}
	Let $\left(\varphi_n\right)$ be an ONS on $[0, 1]$ and $p, q \in E\left(\varphi\right),$ $x \in G,$ where $p(u)=u, q(u)= 1,  u \in \left[0,1\right] .$
	If for $x \in G$
	$$M_n (x)=O(1),$$	
	then for any $f \in C_L,$
	$$S_n(x, f ) = O(1).$$
	
\end{theorem}

\begin{proof}
	If we substitute $F(x)=Q_n (u,x)$  and $f=f'$ in \eqref{eq11}, we obtain
	\begin{align}
		\int_0^1 f'(u)Q_n(u,x)\,du & =n \sum_{i=1}^{n-1} \int_{\frac{i-1}{n}}^{\frac{i}{n}} \left(f'(u)-f'\left(u+\frac{1}{n}\right)\right)\,du \int_0^{\frac{i}{n}}  Q_n(u,x)\,du \nonumber \\
		& \qquad +n\sum_{i=1}^{n}  \int_{\frac{i-1}{n}}^{\frac{i}{n}}  \int_{\frac{i-1}{n}}^{\frac{i}{n}} \left(f'(u)-f'(v)\right)\,dv \; Q_n (u,x)\,du  \nonumber \\ & \qquad +n\int_{1-\frac{1}{n}}^1 f'(u)\,du\; \int_0^1  Q_n (u,x)\,du \nonumber \\
		& =I_1+I_2+I_3. \label{eq13}
	\end{align}
	Due to \eqref{eq3} and the fact that $f'\in C_L$, we get $(\Delta _{in}=[\frac{i-1}{n},\frac{i}{n}])$
	\begin{align}\label{eq14}
		\vert  I_1 \vert &\leq  O\left(\frac{1}{n}\right) n \sum_{i=1}^{n-1} \int_{\Delta_{in}} \left|\int_{0}^{\frac{i}{n}}Q_n\left(u,x\right)du\right| dv \leq \frac{O(1)}{n} \sum_{i=1}^{n-1} \left|\int_{0}^{\frac{i}{n}}Q_n\left(u,x\right)du\right| \\
		& = O\left(1\right) M_n (x) = O(1). \notag
	\end{align}
	Next, since $f'\in C_L$ and $x\in G$, by using inequality  \eqref{eq2.7} from Lemma \ref{l3} and by using H\"older's and Parseval's $\left(\sum\limits_{k=1}^{n} g_k^2(u)\leq 1\right)$ inequalities, we have 
	\begin{align} \label{eq15} 
		\vert  I_2 \vert &\leq O\left(\frac{1}{n^2}\right) n \sum_{i=1}^{n} \int_{\Delta_{in}} \left|Q_n(u,x)\right|du =  O\left(1\right) \frac{1}{n}  \int_{0}^{1} \left| Q_n\left(u,x\right) \right| du \\
		&  = O(1) \frac{1}{n}  \left(\int_{0}^1 Q_n^2(u,x)du \right)^{\frac{1}{2}} = O(1) \frac{1}{n}  \left(\int_{0}^{1} \left(\sum_{k=1}^{n}g_k(u)  \varphi_k(x)\right)^2 du \right)^{\frac{1}{2}} \notag\\
		& =O(1) \frac{1}{n}  \left(\int_{0}^{1} \sum_{k=1}^{n}g_k^2(u)du \sum_{k=1}^{n} \varphi_k^2(x) \right)^{\frac{1}{2}}= O(1) \frac{1}{n}\left(\sum_{k=1}^{n} \varphi_k^2(x) \right)^{\frac{1}{2}}\notag\\
		&= O(1) \left(\frac{1}{n^2}\sum_{k=1}^{n} \varphi_k^2(x) \right)^{\frac{1}{2}} =O(1). \notag
	\end{align}
	Further, in \eqref{eq9} we suppose $f(u) = p(u) = u$ and $q(u)=1,$ $u \in [0,1].$ We receive 
	\begin{eqnarray*}
		S_n(x,p) = \int_{0}^{1}B_n(u,x)du-\int_{0}^{1}Q_n(u,x)du.
	\end{eqnarray*}
	From here as  
	\begin{eqnarray*}
		\int_{0}^{1}B_n(u,x)du=S_n(x,q).
	\end{eqnarray*}	
	we have  (see \eqref{eq5})  
	\begin{eqnarray*}
		S_n(x,p) = S_n(x,q) -\int_{0}^{1}Q_n(u,x)du.
	\end{eqnarray*}
	Consequently  if  $p, q \in E(\varphi),$ we have 
	\begin{eqnarray}\label{eq161}
		\vert  I_3 \vert &\leq& n\int_{1-\frac{1}{n}}^1  \vert  f'(u)\vert \,du \;\bigg\vert \int_0^1  Q_n (u,x)\,du\bigg\vert \leq n \ \frac{1}{n} \ \max_{i \in [0,1]} \vert  f'(u)\vert \bigg\vert \int_0^1  Q_n (u,x)\,du\bigg\vert = O(1).
	\end{eqnarray}
	Taking the evaluations of $I_1,$ $I_2,$ and $I_3$ into account in \eqref{eq13} we conclude
	\begin{equation}\label{eq17}
		\bigg\vert  \int_0^1  f' (u) Q_n (u,x)\,du\bigg\vert  = O(1).
	\end{equation}
	Finally, according to the conditions of this theorem and \eqref{eq8}, we can deduce that
	$$ S_n (x,f) =O(1).$$
	Theorem \ref{t2} is completely proved.
	
\end{proof}

\begin{theorem}\label{t3}
	Let $(\varphi_n)$ be an ONS on $[0,1]$. If for some $t\in G$,
	$$
	\limsup_{n\to \infty} M_n (t)=+\infty,
	$$
	then there exists a function $r'\in C_L$ such that
	$$
	\limsup_{n\to \infty} \vert S_n (t,r)\vert =+\infty  .
	$$
\end{theorem}

\begin{proof}
	Firstly, according to Lemma \ref{l2} and Lemma \ref{l5}
	\begin{eqnarray}
		\label{eq3.4}
		\left| \int_0^1 B_n (u,t)\,du\right| =O(1) \text{ and } 
		\left| \int_0^1 Q_n (u,t)\,du\right| =O(1)
	\end{eqnarray}
	otherwise, Theorem \ref{t3} is proved.
	
	We defined the sequence of functions $(f_n)$ as follows:
	\begin{eqnarray} \label{eq22}
		f_n(u)=\int_{0}^{u} sign \int_{0}^{y} Q_n(v,t)dv dy, \ \ \ n=1,2, \dots .
	\end{eqnarray}
	In \eqref{eq161} we substitute $Q_n\left(u,x\right)=Q_n\left(u,t\right)$ and $f'(u)= f_n (u),$ then
	\begin{eqnarray}
		\label{eq23} \int_{0}^{1} f_n (u)Q_n (u, t) du
		&=& \sum_{i=1}^{n-1}\left(f_n\left(\frac{i}{n}\right)-f_n\left(\frac{i+1}{n}\right)\right)\int_{0}^{\frac{i}{n}}Q_n(u,t)du \notag   \\
		&+& \sum_{i=1}^{n}\int_{\frac{i-1}{n}}^{\frac{i}{n}}\left(f_n\left(u\right) -f_n\left(\frac{i}{n}\right)\right)Q_n(u,t) du    \\
		&+& f_n\left(1\right)\int_{0}^{1}Q_n(u,t)du=S_1 +S_2  +S_3.    \notag
	\end{eqnarray}
	By using \eqref{eq22}, Lemma \ref{l3} and  H\"older's and Cauchy's inequality, we receive (see $|I_2|$)
	\begin{eqnarray} \label{eq27}
		\left|S_2\right| &\leq& \frac{1}{n} \sum_{i=1}^{n} \int_{\frac{i-1}{n}}^{\frac{i}{n}}\left|Q_n(u,t)\right|du = \frac{1}{n}\int_{0}^{1}\left|Q_n(u,t)\right|du \\
		&\leq&  O(1) \frac{1}{n} \left(\int_{0}^{1}  \left(\sum_{k=1}^{n}g_k(u)\varphi_{k}(t)\right)^2 du\right)^{\frac{1}{2}} = O(1). \notag
	\end{eqnarray}
	
	Afterwards taking into account \eqref{eq3.4} and \eqref{eq22} we get
	\begin{eqnarray}
		\label{eq28}
		\left|S_3\right| = O(1)\left|\int_{0}^{1}Q_n(u,t)du\right| = O(1).
	\end{eqnarray}

	Let $D_n$ be a set of all $i$ $\left(i=1,2,...,n-1\right)$, for all of which, there exists a point $t\in\left[\frac{i-1}{n},\frac{i}{n}\right]$ such, that 
	\begin{equation}
		sign\int_{0}^{\frac{i}{n}}Q_n \left(u,t\right) du \ne sign\int_{0}^{t}Q_n \left(u,t\right) du. \label{2}
	\end{equation}
	
	Suppose that $i\in D_n.$ On account  continuity of  function  $\int_{0}^{t}Q_n \left(u,x\right) du$ on $[0, 1]$ for some $t_{i_n}\in\left[\frac{i}{n},\frac{i+1}{n}\right]$ we have
	\begin{equation*}
		\int_{0}^{t_{i_n}}Q_n \left(u,t\right) du=0.
	\end{equation*}
	Consequently,
	\begin{eqnarray}
		\left|\int_{0}^{\frac{i}{n}}Q_n \left(u,t\right) dx\right|= \left|\int_{0}^{t_{i_n}}Q_n \left(u,t\right) du +\int_{t_{i_n}}^{\frac{i}{n}}Q_n \left(u,t\right) du\right| \leq\int_{t_{i_n}}^{\frac{i}{n}} \left|Q_n \left(u,t\right)du\right|.   \notag
	\end{eqnarray}
	Further (see \eqref{eq2.7}), by using H\"older's inequality, we have
	\begin{eqnarray*}
		&& \sum_{i\in D_n}\left|\int_{0}^{\frac{i}{n}}Q_n \left(u, t\right)du \right| \leq	\sum_{i=1}^{i-1}\left|\int_{t_{i_n}}^{\frac{i}{n}}Q_n \left(u, t\right)du \right| \leq \int_{0}^{1}\left|Q_n \left(u,t\right)\right|du \\
		&& \qquad \leq \left(\int_{0}^{1}Q_{n}^{2} \left(u, t\right) du\right)^{\frac{1}{2}}=\left(\int_{0}^{1}\left(\sum_{k=1}^{n}g_k(u)\varphi_{k}(t)\right)^2 du\right)^{\frac{1}{2}}  \\
		&& \qquad = \left(\int_{0}^{1}\sum_{k=1}^{n}g_k^2(u) du \sum_{k=1}^{n}\varphi_{k}^2(t) \right)^{\frac{1}{2}} \leq \left(\sum_{k=1}^{n}\varphi_{k}^2(t) \right)^{\frac{1}{2}}.
	\end{eqnarray*}
	Then (see Lemma \ref{l1})
	\begin{eqnarray}
		\frac{1}{n} \sum_{i\in D_n}\left|\int_{0}^{\frac{i}{n}}Q_n \left(u, t\right)du \right| \leq O(1)\frac{1}{n}\left( \sum_{k=1}^{n} \varphi_k^2(t)\right)^{\frac{1}{2}} = O(1).\label{eq3.10}
	\end{eqnarray}
	
	At present we denote $F_n =\{1, 2, \dots n-1 \} \setminus D_n.$ Suppose that $i\in F_n.$ then according to definition of $D_n$  we have 
	\begin{eqnarray*}
		\left(f_n\left(\frac{i}{n}\right)-f_n\left(\frac{i+1}{n}\right)\right)\int_{0}^{\frac{i}{n}}Q_n(u,t)du =-\frac{1}{n} \left|\int_{0}^{\frac{i}{n}}Q_n(u,t) du\right|.
	\end{eqnarray*}
	According to last equality, we have
	\begin{eqnarray}
		\label{eq24} 
		\left|\sum_{i\in F_n} \left(f_n \left(\frac{i}{n}\right) -f_n \left(\frac{i+1}{n}\right)\right)\int_{0}^{\frac{i}{n}}  Q_n(u,t)du\right| = \frac{1}{n} \sum_{i\in F_n} \left| \int_{0}^{\frac{i}{n}} Q_n(u,t)du\right|.   
	\end{eqnarray}
	
	According to  \eqref{eq3.10} and \eqref{eq27} we obtain 
	\begin{eqnarray} \label{eq151}
		\left|S_1\right| &&= \left| \sum_{i=1}^{n-1}\left(f_n\left(\frac{i}{n}\right) -f_n\left(\frac{i+1}{n}\right)\right)\int_{0}^{\frac{i}{n}}Q_n(u,t)du \right|   \\ 
		&& \geq \left|\sum_{i\in F_n} \left(f_n\left(\frac{i}{n}\right) - f_n\left(\frac{i+1}{n}\right)\right) \int_{0}^{\frac{i}{n}}Q_n(u,t)du \right| \notag \\
		&& -  \left|\sum_{i\in D_n} \left(f_n\left(\frac{i}{n}\right) - f_n\left(\frac{i+1}{n}\right)\right) \int_{0}^{\frac{i}{n}}Q_n(u,t)du \right| \notag \\
		&& \geq  \frac{1}{n} \sum_{i\in F_n} \left|\int_{0}^{\frac{i}{n}}Q_n(u,t)du\right|- \frac{1}{n} \sum_{i\in D_n} \left|\int_{0}^{\frac{i}{n}}Q_n(u,t)du\right| \notag \\
		&& = \frac{1}{n} \sum_{i=1}^{n-1} \left|\int_{0}^{\frac{i}{n}}Q_n(u,t)dt\right|- \frac{2}{n} \sum_{i\in D_n} \left|\int_{0}^{\frac{i}{n}}Q_n(u,t)du\right|  \notag \\
		&&\geq M_n(t)- O (1).\notag
	\end{eqnarray}
	Finally, from \eqref{eq23} using \eqref{eq27}, \eqref{eq28} and \eqref{eq151},  we have
	\begin{eqnarray*}
		\left|\int_{0}^{1}f_n(u)Q_n(u,t)du\right| \geq M_n(t)-O(1).
	\end{eqnarray*}
	The condition of Theorem  \ref{t3} imply 
	\begin{eqnarray}	\label{eq29} 
		\limsup_{n \to \infty} \left|\int_{0}^{1}f_n(u)Q_n(u,t)du\right|=+\infty.
	\end{eqnarray}
	Consider the sequence of liner and bounded functionals on the Banach space $Lip1$
	\begin{eqnarray*}
		U_n(f) = \int_{0}^{1}f_n(u)Q_n(u,t)du.
	\end{eqnarray*}
	By \eqref{eq29}
	\begin{eqnarray*}
		\limsup_{n \to \infty}\left|U_n(f_n) \right|= + \infty.
	\end{eqnarray*}
	On the other hand                
	\begin{eqnarray*}
		\left|\left|f_n\right|\right|_{Lip1}=	\left|\left|f_n\right|\right|_{C} + \sup_{x,y\in \left[0,1\right]} \frac{\left|f_n(x)-f_n(y)\right|}{\left|x-y\right|} \leq 2. 
	\end{eqnarray*}
	Consequently (see \eqref{eq29}), according to the Banach-Steinhaus theorem, there exist such a function $h\in Lip1$ that
	\begin{eqnarray}	\label{eq30} 
		\limsup_{n \to \infty} \left|\int_{0}^{1}h(u)Q_n(u,t)du\right|=+\infty.
	\end{eqnarray}
	Let
	\begin{eqnarray*}
		m(u)=\int_{0}^{u}h(v)dv,
	\end{eqnarray*}
	using lemma \ref{l4} we get
	\begin{eqnarray*}
		S_n (t,m)= m(1) \int_0^1  B_n (u,t)\,du-\int_0^1  h(u) Q_n (u,t)\,du.
	\end{eqnarray*}
	From \eqref{eq3.4} and \eqref{eq30}  we get
	\begin{eqnarray*}
		\limsup_{n \to \infty}\left|S_n(t,m)\right|=+\infty.
	\end{eqnarray*}
	As $m' = h \in Lip1$ Theorem \ref{t3} is proved.
	
\end{proof}

Now we show that the condition of Theorem \ref{t2}  $( q,p \in E(\varphi), \ q(u)=1, \ p(u)=u, \ u\in [0,1])$ don't guarantee that
\begin{eqnarray*}
	\limsup_{n \to \infty}\left|S_n(t,f)\right|<\infty,
\end{eqnarray*}
for any function $f \in C_L.$

Indeed
\begin{theorem}
	\label{theorem4} 
	There exists a function $g \in C_L$ and ONS $(G_n)$  such that 
	\begin{eqnarray*}
		\int_{0}^{1} G_n (u) du =0, \ \ \ \   \int_{0}^{1} u \ G_n (u) du =0,  \ \ \ n=1, 2, \dots, 
	\end{eqnarray*}
	and
	\begin{eqnarray*}
		\limsup_{n \to \infty}\left|S_n(x,g,G)\right|= \limsup_{n \to \infty}\left|	\sum_{k=1}^{n} C_n(g,G) G_n(x)\right|=+ \infty, 
	\end{eqnarray*}
	where 
	\begin{eqnarray*}
		C_n(g,G)= \int_{0}^{1} g(u) G_n (u) du =0,  \ \ \ (n=1, 2, \dots)
	\end{eqnarray*}
	and
	\begin{eqnarray*}
		S_n(x,g,G) = 	\sum_{k=1}^{n} C_n(g,G) G_n(x).
	\end{eqnarray*}

\end{theorem}

\begin{proof}
	Let us assume that $f(x)=1-\cos 4(u- 1/2) \pi.$ According to the Banach Theorem there exist an ONS $(\varphi_n),$ such that a.e. on $[0,1]$
	\begin{eqnarray}
		\label{eq31} \limsup_{n \rightarrow +\infty} \left|S_n (x, f, \varphi)\right| =+\infty.
	\end{eqnarray}
	The system   $\varPhi_n(u)$  we define as follows
	
	\begin{equation*}
		\varPhi_n\left( u\right) =\left\{ 
		\begin{array}{ccc}
			\varphi_n(2u), & \text{when} & u\in \left[ 0,\frac{1}{2}\right), \\ 
			\\
			-\varphi_n\left(2\left(u-\frac{1}{2}\right)\right), & \text{when} & u\in \left[ \frac{1}{2},1\right].
		\end{array}%
		\right.
	\end{equation*}
	
	It is easy to prove that $\left(\varPhi_n\right)$ is an ONS on $[0,1]$ and $\int_{0}^1 \varPhi_n (u) du =0 , \ n=0,1, \dots .$
	
	Now we investigate the function
	\begin{equation*} \label{eq33}
		g\left( u\right) =\left\{ 
		\begin{array}{ccc}
			f(2u), & \text{when} & u\in \left[ 0,\frac{1}{2}\right), \\ 
			\\
			0, & \text{when} & u\in \left[ \frac{1}{2},1\right].
		\end{array}%
		\right.
	\end{equation*}
	We have
	\begin{eqnarray*}
		C_n(g,\varPhi)&=& \int_{0}^{1} g(u) \varPhi_n (u) du =\int_{0}^{\frac{1}{2}} f\left(2u\right) \varphi_n\left(2u\right)du = \frac{1}{2}\int_{0}^{1} f\left(u\right) \varphi_n\left(u\right)du = \frac{1}{2} C_n\left(f, \varphi\right).
	\end{eqnarray*}
	Consequently (see \eqref{eq31})
	\begin{eqnarray}
		\label{eq3.15} \limsup_{n \rightarrow +\infty} \left|S_n (x, f, \varPhi)\right| = \frac{1}{2} \limsup_{n \rightarrow +\infty} \left|S_n (x, f, \varphi)\right|=+\infty.
	\end{eqnarray}
	
	Now we define the next ONS $G_n(u)$ as follows
	
	\begin{equation} \label{eq34}
		G_n\left( u\right) =\left\{ 
		\begin{array}{ccc}
			\varPhi_{n}(2u), & \text{when} & u\in \left[ 0,\frac{1}{2}\right), \\ 
			\\
			- \varPhi_{n} \left(2\left(u-\frac{1}{2}\right)\right), & \text{when} & u\in \left[ \frac{1}{2},1\right].
		\end{array}%
		\right.
	\end{equation}
	After we define 	
	\begin{equation*} \label{eq33}
		h\left( u\right) =\left\{ 
		\begin{array}{ccc}
			g(2u), & \text{when} & u\in \left[ 0,\frac{1}{2}\right), \\ 
			\\
			0, & \text{when} & u\in \left[ \frac{1}{2},1\right].
		\end{array}%
		\right.
	\end{equation*}
	We have
	\begin{eqnarray*}
		C_n(h,G)&=& \int_{0}^{1} h(u) G_n (u) du =\int_{0}^{\frac{1}{2}} g\left(2u\right) \varPhi_n\left(2u\right)du\\
		&=& \frac{1}{2}\int_{0}^{1} g\left(u\right) \varPhi_n\left(u\right)du = \frac{1}{2} C_n\left(g, \varPhi\right).
	\end{eqnarray*}
	Such we have that $h'\in Lip1$ and  (see \eqref{eq31})
	\begin{eqnarray}
		\label{eq3.16} \limsup_{n \rightarrow +\infty} \left|S_n (x, h, G)\right| = \frac{1}{2} \limsup_{n \rightarrow +\infty} \left|S_n (x, g, \varPhi)\right|=+\infty.
	\end{eqnarray}
	Let $p(u)= u  ,$ we get
	\begin{eqnarray*}
		C_n(p,G)&=& \int_{0}^{1} u\ G_n (u) du =\int_{0}^{\frac{1}{2}}u\ \varPhi_n\left(2u\right)du - \frac{1}{2} \int_{0}^{1} u\ \varPhi_{n}\left(2\left(u-\frac{1}{2}\right)\right) du  \\
		&=&\frac{1}{4}\int_{0}^{1} u\ \varPhi_n\left(u\right)du - \frac{1}{2}\int_{0}^{1} \left({u + \frac{1}{2}}\right) \varPhi_n\left(u\right)du = -\frac{1}{4} C_n\left(q, \varPhi\right)=0.
	\end{eqnarray*}
	Finally we receive \\
	1) If $q(u)=1$ and $p(u)=u$ when $u \in [0,1],$ then
	\begin{eqnarray*}
		C_n\left(p,G\right) = C_n\left(q, G\right) = 0 \text{ or } 	\limsup_{n \rightarrow +\infty} \left|S_n (x, p, G)\right| = \frac{1}{2} \limsup_{n \rightarrow +\infty} \left|S_n (x, q, G)\right|<+\infty.
	\end{eqnarray*}	
	And (see \eqref{eq3.16})\\
	2) $\limsup\limits_{n \rightarrow +\infty} \left|S_n (x, h, G)\right| = +\infty \text{ where } h' \in Lip1. $
	
	Theorem \ref{theorem4} is completely proved.
	
\end{proof}

\section{Problems of efficiency} \label{sec4}
\begin{theorem}
	\label{t5}
	Let $\varphi_{n}(u) = \sqrt{2} \cos 2 \pi n u,$ then for any $x\in [0,1]$
	
	\begin{eqnarray*}
		M_n(x)= O (1).
	\end{eqnarray*}
\end{theorem}
\begin{proof}
	Indeed
	\begin{eqnarray*}
		M_n(x) &=& \frac{1}{n} \sum_{i=1}^{n-1} \left|\int_{0}^\frac{i}{n} Q_n(u,x)du\right|=\frac{2}{n} \sum_{i=1}^{n-1} \left|\int_{0}^\frac{i}{n} \sum_{k=1}^{n} \int_{0}^{u} \cos 2 \pi k v dv du \cos 2 \pi k x \right| \\
		&=& O(1) \max_{1 \leq i \leq n} \left|\int_{0}^\frac{i}{n} \frac{1}{2 \pi} \sum_{k=1}^{n}  \frac{1}{k} \sin 2 \pi k u du \cos 2 \pi k x \right| = O(1) \left(\sum_{k=1}^{n} \frac{1}{k^2}\right)^{\frac{1}{2}} = O(1).
	\end{eqnarray*}
	
\end{proof}

\begin{theorem}
	\label{t6}
	Let $(X_n)$ be Haar system (see \cite{Alexits},Ch.2) then for any $x \in [0,1]$
	\begin{eqnarray*}
		M_n(x) = O(1).
	\end{eqnarray*}
	
\end{theorem}
\begin{proof}
	If $2^s<m\leq2^{s+1}$ then 
	\begin{eqnarray*}
		\left|\int_{0}^{u}X_m(v) dv\right| \leq 2^{-\frac{s}{2}}
	\end{eqnarray*}
	and $(i=1,2, \dots ,n-1)$
	\begin{eqnarray*}
		\left| \int_{0}^{\frac{i}{n}} \int_{0}^{u} \sum_{m=2^s +1}^{2^{s+1}} X_m(v)dvX_m(x)du\right| \leq 2 \ 2^{-s}.
	\end{eqnarray*}	
	Finally 
	\begin{eqnarray*}
		M_n(x) &=& \frac{1}{n} \sum_{i=1}^{n-1} \left|\int_{0}^\frac{i}{n} Q_n(u,x)du\right|  = O(1)  \max_{1 \leq i \leq n} \left|\int_{0}^\frac{i}{n} \sum_{s=0}^{d} \int_{0}^{u} \sum_{m=2^s +1}^{2^{s+1}} X_m(v)dvX_m(x)du\right| = O(1).
	\end{eqnarray*}

	From here it is evident that Theorem \ref{t6} is valid.
	
\end{proof}

\section{Conclusion}\label{sec13}

Based on the discussion presented in this article, it is evident that although Fourier partial sums related to general orthonormal systems do not converge for all functions $f$ classified under a differentiable class of $ C_L$ functions, we can identify a specific subset of orthonormal systems. These subsets include functions that meet particular criteria, ensuring that the Fourier partial sums for $ C_L$ class functions exhibit convergence (refer to Theorem \ref{t2}). Moreover, we have established that the criteria applied to the functions of the orthonormal systems are both precise and reliable.

Additionally, it is important to note that each orthonormal system encompasses a subsystem for which the general Fourier series of any function $ f$ within the $ C_L$ class converges almost everywhere on the interval $[0,1].$

\section{Manuscript processing} 

The \textbf{evaluation process} varies from journal to journal.

\begin{itemize}
	\item Single-blind review: the reviewers remain anonymous to the authors.
	\item Double-blind review: the reviewers do not know who the authors are, nor do the authors know who has evaluated their manuscript.
\end{itemize}
Typically, at least two independent experts are invited to review a manuscript’s content. The manuscript is then either accepted, rejected, or returned for revision based on their evaluation.

With many journals, you can propose reviewers who come from outside of your closest areas of academia. It is at the editors’ discretion whether to accept these proposals.

\textbf{Galley proofs}: Before your contribution is published, you will receive a proof of the article to proofread. At this point in the publication process, there must be no more changes made to the content: only minor corrections in form and phrasing are possible.

\begin{acknowledgement}
The authors would like to express their sincere gratitude to their colleagues and collaborators for their insightful discussions, constructive suggestions, and continuous support throughout the development of this work.

We would also like to express our appreciation in advance to the anonymous reviewers for their careful reading of the manuscript and for any thoughtful comments that may help improve the clarity and quality of the paper.

Furthermore, we extend our thanks to the editorial team and reviewers of the \emph{Georgian Mathematical Journal} for their commitment to maintaining high academic standards, which continues to support the advancement of research in the field.
\end{acknowledgement}

\begin{funding}
The research of third author is supported by Shota Rustaveli National Science Foundation grant no. FR-24-698;
\end{funding}

\end{document}